\documentclass[12pt]{article}
\usepackage{stmaryrd}
\usepackage{bbm}
\usepackage{amsfonts}
\usepackage{amsmath}
\usepackage{amssymb}
\usepackage{amsthm}
\usepackage{mathrsfs}
\usepackage{accents}
\usepackage{graphicx}
\usepackage[left=1in,top=1in,right=1in]{geometry}

\numberwithin{equation}{section}

\newtheorem{thm}{Theorem}[section]
\newtheorem{lem}{Lemma}[section]
\newtheorem{rem}{Remark}[section]
\newtheorem{defi}{Definition}[section]
\newtheorem{prop}[thm]{Proposition}
\newtheorem{cor}[thm]{Corollary}

\newtheorem{question}[thm]{Question}

\begin{document}

\title{Superlinearity of geodesic length in 2$D$ critical first-passage percolation}
\author{Michael Damron \thanks{The research of M. D. is supported by NSF grant DMS-0901534 and an NSF CAREER award.} \\ \small{Georgia Tech}  \and Pengfei Tang\\ \small{Indiana University, Bloomington}}

\maketitle

\begin{abstract}
First-passage percolation is the study of the metric space $(\mathbb{Z}^d,T)$, where $T$ is a random metric defined as the weighted graph metric using random edge-weights $(t_e)_{e\in \mathcal{E}^d}$ assigned to the nearest-neighbor edges $\mathcal{E}^d$ of the $d$-dimensional cubic lattice. We study the so-called critical case in two dimensions, in which $\mathbb{P}(t_e=0)=p_c$, where $p_c$ is the threshold for two-dimensional bond percolation. In contrast to the standard case $(<p_c)$, the distance $T(0,x)$ in the critical case grows sub linearly in $x$ and geodesics are expected to have Euclidean length which is superlinear. We show a strong version of this super linearity, namely that there is $s>1$ such that with probability at least $1-e^{-\|x\|_1^c}$, the minimal length geodesic from $0$ to $x$ has at least $\|x\|_1^s$ number of edges. Our proofs combine recent ideas to bound $T$ for general critical distributions, and modifications of techniques of Aizenman-Burchard to estimate the Hausdorff dimension of random curves.
\end{abstract}

\section{Introduction}

\subsection{Main result}

We study critical first-passage percolation (FPP) in two dimensions. This is a special case of general FPP, which is a stochastic growth model introduced in the '60s by Hammersley and Welsh \cite{HW}. The setup is as follows: on the square lattice $\mathbb{Z}^2$ with nearest-neighbor edges $\mathcal{E}^2$, we assign i.i.d. nonnegative passage times $(t_e)_{e \in \mathcal{E}^2}$ to the edges and define the induced random metric
\[
T(x,y) = \inf_{\pi : x \to y} T(\pi),~ \text{for } x,y \in \mathbb{Z}^2,
\]
where the infimum is over all lattice paths $\pi$ from $x$ to $y$, and $T(\pi) = \sum_{e \in \pi}t_e$. If none of the $t_e$'s are zero, then $T$ is a metric (generally a pseudometric), and FPP is the study of geometric and probabilistic properties of the metric space $(\mathbb{Z}^2, T)$.

Instead of assuming that no $t_e$'s are zero, one typically assumes that the common distribution function $F$ of the weights does not give too much mass to zero: $F(0) < 1/2$, as $1/2$ is the critical threshold for Bernoulli percolation in two dimensions, and this ensures that a.s. there is no infinite component of zero-weight edges (edges which we can traverse in zero time). Under this assumption and some mild integrability constraint, there is a type of law of large numbers for $T$, called the shape theorem, which states that $T(0,x)$ grows linearly as $\|x\|_1 \to \infty$: there is a deterministic norm $g$ on $\mathbb{R}^2$ such that a.s.,
\begin{equation}\label{eq: shape}
\limsup_{\|x\|_1 \to \infty} \frac{|T(0,x) - g(x)|}{\|x\|_1} = 0.
\end{equation}
So on large scales, the metric $T$ is comparable to the Euclidean one.

The focus of this paper is geodesics, specifically their (Euclidean) lengths. A geodesic from $x$ to $y$ is a minimizer for $T$: a lattice path $\pi$ from $x$ to $y$ with $T(\pi) = T(x,y)$. It has been shown \cite{WR} that for any $F$, a.s. there is a geodesic between any $x$ and $y$, although uniqueness of geodesics is equivalent to continuity of $F$. In the general case stated above, $F(0) < 1/2$, it is known that the comparability of $T$ to the Euclidean norm extends in a sense to geodesics, which have a linear number of edges. Specifically, \cite[Theorem~4.6]{ADH15} if $F(0)<1/2$, then there are $c_1,c_2>0$ such that for any $x$,
\begin{equation}\label{eq: stretched}
\mathbb{P}(m(x) \geq c_1\|x\|_1) \leq c_1e^{-c_2 \sqrt{\|x\|_1}},
\end{equation}
where $m(x)$ is the maximal number of edges in any geodesic from 0 to $x$.

If $F(0) > 1/2$, there is a.s. an infinite component of zero-weight edges, and $T(0,x)$ is stochastically bounded in $x$, so the function $g$ in \eqref{eq: shape} is identically zero. In this case, one can also show that the minimal length geodesic between two points has a linear number of edges \cite[Theorem~4]{Zhang95}. The so-called critical case, when $F(0)=1/2$, is considerably more complicated, and has only recently been significantly explored. Although there is no infinite cluster of zero-weight edges, the clusters are large enough to force $g \equiv 0$. The precise behavior of $T(0,x)$ as $\|x\|_1 \to \infty$ was quantified in \cite{DLW}, with necessary and sufficient conditions on $F$ for stochastic boundedness of $T(0,x)$ in $x$ (and whether boundedness indeed holds depends on $F$ in the critical case, as discovered by Zhang \cite{Zhang99}). Because geodesics can take paths in large critical zero-weight clusters, and these clusters have irregular structure, Kesten \cite[p.~259]{aspects} was led to ask a version of the following (see also \cite[p.~1029]{SZ}):

\begin{question}\label{q: KZ}
In the critical case, is there $s>1$ such that a.s., $N_{0,x} \leq \|x\|_1^s$ holds for only finitely many $x \in \mathbb{Z}^2$?
\end{question}

In this paper, we give a positive answer to this question, with a stretched exponential estimate similar to \eqref{eq: stretched}. For any $x,y \in \mathbb{Z}^2$, let $N_{x,y}$ be the minimal number of edges in an geodesic between $x$ and $y$. From this point forward, we assume that $F$ is critical:
\begin{equation}\label{eq: time_constraints}
F(0^-)=0 \text{ and } F(0) = 1/2.
\end{equation}
\begin{thm}\label{thm: main_thm}
Assuming \eqref{eq: time_constraints}, there exist $c>0,s>1$ such that for all nonzero $x \in \mathbb{Z}^2$,
\begin{equation}\label{eq: main_thm_bound}
\mathbb{P}(N_{0,x} \leq \|x\|_1^s) \leq (1/c) \exp\left( -\|x\|_1^c \right).
\end{equation}
\end{thm}

We end this section with various remarks about the main result. First, our proof works on lattices where near-critical percolation estimates hold, and this includes most regular planar lattices, including edge or site FPP on the hexagonal lattice, square lattice, triangular lattice, etc. The main difference is that one must assume that $F(0) = p_c$, where $p_c$ is the critical threshold for percolation on that lattice. Next, it is important to point out that using the work of Aizenman-Burchard, we can give a simple proof for Question~\ref{q: KZ} (see Section~\ref{sec: sketch}), but this argument only gives a small polynomial decay of the probability in \eqref{eq: main_thm_bound}. Our main inequality is sufficient (but polynomial decay is not) to, for example, bound the length of all geodesics simultaneously between points sufficiently far apart in a box.



Last we briefly remark on the proof; a full outline of it appears in Section~\ref{sec: sketch}. The strategy is to combine a block argument from Pisztora \cite{P} with the Aizenman-Burchard technique. The block argument is quite similar to \cite{P}. The main difficulty is in the Aizenman-Burchard technique: the hypothesis of \cite{AB} does not hold for our model. Their method relies on a strong independence assumption: to apply their theorem one would need to know that there is $\rho<1$ such that for any number of thin cylinders $C_1, \ldots, C_k$ which are sufficiently separated, the probability that a geodesic crosses all of these cylinders in the long direction (has a ``straight run'' in each cylinder) is at most $\rho^k$. Under this assumption, we could copy their arguments to conclude that straight runs are sufficiently sparse globally to deduce a superlinear lower bound for geodesic length. Although FPP is built on i.i.d. weights, segments of geodesics are highly correlated, and such an independence assumption is not obviously true (and is actually false in the non-critical case). So we need to show differently that hierarchies of nested cylinders cannot contain too many straight runs by geodesics. The approach is to show that if geodesics do cross too many such cylinders, there is a high probability that many of these cylinders are ``slow'' (in a sense described by near-critical percolation paths) and force the passage time of geodesics to be large. We combine this with new upper bounds on passage times of geodesic segments using ideas from \cite{DLW} to conclude that straight runs are sparse.

\subsection{Notation and tools from percolation}

We will couple the FPP model to various percolation models. To do this, we let $(\omega_e)$ be a collection of i.i.d. uniform $(0,1)$ random variables and $t_e = F^{-1}(\omega_e)$, where $F^{-1}$ is the generalized inverse
\[
F^{-1}(t) = \sup\{s : F(s) < t\},~ t \in (0,1).
\]
Then the variables $(t_e)$ are i.i.d. with distribution $F$. For $p \in [0,1]$, an edge $e$ is called $p$-open if $\omega_e \leq p$ and $p$-closed otherwise. A path is a sequence of edges (or their endpoints, or both) such that each consecutive pair of edges shares an endpoint, and a circuit is a path which starts and ends at the same point. If $\Gamma$ is a path, we write $\#\Gamma$ for the number of edges in it. For $n \geq 1$, the box $B(n)$ is defined as $[-n,n]^2$.

Next we define the dual lattice, which is used in Section~\ref{sec: block}. It is $(\mathbb{Z}^2)^* = \mathbb{Z}^2 + (1/2,1/2)$ with its set of nearest-neighbor edges $(\mathcal{E}^2)^*$. An edge $e$ has exactly one dual edge $e^*$ that bisects it. We define variables $(\omega_{e^*})$ by the rule $\omega_{e^*} = \omega_e$ and correspondingly use the terms $p$-open and $p$-closed. Thus a $p_c$-closed dual path is a path of edges $e^*$ on the dual lattice each with $\omega_{e^*} > p_c$.

Last we give properties of correlation length, which will be vital for our work. For $\epsilon>0$ and $p > p_c$, we define
\[
L(p,\epsilon) = \min\{ m \geq 1 : \mathbb{P}(\sigma(p,m,m)) > 1-\epsilon\},
\]
where $\sigma(p,m,m)$ is the event that the box $B(m)$ has a $p$-open left-right crossing. This is a path, all of whose edges are $p$-open and in $B(m)$, which touches the left and right sides of the box. It is shown in \cite[Eq.~(1.24)]{kesten_scaling_relations} that for some $\epsilon_0$ and any $\epsilon_1, \epsilon_2 \in (0,\epsilon_0]$, one has
\[
L(p,\epsilon_1) \asymp L(p,\epsilon_2) \text{ as } p \downarrow p_c,
\]
so we just set $L(p) = L(p,\epsilon_0)$. (The notation $\asymp$ means that $L(p,\epsilon_1)/L(p,\epsilon_2)$ is bounded away from 0 and $\infty$ as $p \downarrow p_c$.) We will use the following properties of correlation length. Setting
\[
p_m = \min\{p : L(p) \leq m\},
\]
\begin{itemize}
\item (see \cite[Eq.~(2.10)]{jarai_03}) there exists $c_1 \in (0,1)$ such that for all $m \geq 1$,
\begin{equation}\label{eq: L_p_n}
c_1 m \leq L(p_m) \leq m,
\end{equation}
\item (see \cite[Sec.~2.1]{DSV}) for positive integers $k$ and $l$, there exists $\delta_{k,l}>0$ such that for any positive integer $n$ and for all $p \in [p_c,p_n]$,
\[
\mathbb{P}(\text{there is a } p\text{-open horizontal crossing of }[0,kn] \times [0,ln]) > \delta_{k,l}
\]
and
\begin{equation}\label{eq: closed_dual_p_n}
\mathbb{P}(\text{there is a }p\text{-closed horizontal dual crossing of } [0,kn] \times [0,ln]) > \delta_{k,l}.
\end{equation}
Using these inequalities and a gluing construction involving the FKG inequality \cite[Theorem~2.4]{grimmett}, one can construct open or closed circuits in annuli around 0. For example, uniformly in $n$ and $p \in [p_c,p_n]$,
\[
\mathbb{P}(\text{there is a }p\text{-open circuit around }0 \text{ in }B(2n) \setminus B(n)) > 0.
\]
We will use similar statements throughout the paper. See \cite{grimmett} for the relevant techniques and background.
\item in the proof of  Lemma~\ref{lem: kiss_lem} we will use ``quasi-multiplicativity'' of certain events. Generally this means that probabilities of arm events factor up to a constant. We will use this for 3-, 4-, and 5-arm events in both the full- and half-planes, for near critical values of $p$. One example of this property is the following: for $n_1 \leq n_2$, and $p,q$ with $L(p),L(q) \geq n_2$, let $A(n_1,n_2,p,q)$ be the event that there are two $p$-open disjoint paths from $\partial B(n_1)$ to $\partial B(n_2)$ and two $q$-closed dual paths from $\partial B(n_1)$ to $\partial B(n_2)$ so that the open and closed paths alternate. Then there is a constant $c$ such that for all $0 \leq n_1 \leq n_2 \leq n_3$ and $q,p>p_c$ with $L(p),L(q) \geq n_3$,
\begin{equation}\label{eq: QM}
\mathbb{P}(A(n_1,n_3,p,q)) \geq c \mathbb{P}(A(n_1,n_2,p,q)) \mathbb{P}(A(n_2,n_3,p,q)).
\end{equation}
(See \cite{nolin} for more background on arm events and their properties, like quasi-multiplicativity.)
\end{itemize}

\subsection{Sketch of proof and outline of paper}\label{sec: sketch}

\subsubsection{Argument for Question~\ref{q: KZ}}
We begin with a simple proof of the following statement: there exists $s>1$ such that a.s.,
\begin{equation}\label{eq: easy_a_s}
N_{0,x} \leq \|x\|_1^s \text{ for only finitely many } x \in \mathbb{Z}^2.
\end{equation}
This essentially follows from the FKG inequality, the Russo-Seymour-Welsh theorem \cite[Ch.~11]{grimmett} and the work of Aizenman-Burchard. The latter implies (see a discussion in \cite[p.~3597]{DHS_RW}) that there is $s>1$ such that
	\begin{equation}\label{equ:long_open_path}
		\lim_n \mathbb{P}\left(\exists \text{ a }p_c\text{-open path } \Gamma\text{ crossing } B(3^4n) \setminus B(n) \text{ with } \# \Gamma \leq n^s\right)=0.
	\end{equation}

	

Next, for $k \geq 0$, let $E_k$ be the event that the follow conditions hold: putting $Ann_m = B(3^{m+1}) \setminus B(3^m)$,
\begin{enumerate}
\item there is an $p_c$-open circuit around 0 in $Ann_{3k}$,
\item there is an $p_c$-open circuit around 0 in $Ann_{3k+2}$,
\item  there is an  $p_c$-open path crossing the rectangle $[3^{3k},3^{3k+3}]\times[-3^{3k},3^{3k}]$ from the left side to the right, and
\item any $p_c$-open path $\Gamma$ crossing the annulus $Ann_{3k+1}$ satisfies $ \# \Gamma \geq 3^{3ks}$.
\end{enumerate}
(See Figure~\ref{fig: a_s_event} for an illustration.) By the RSW theorem, the FKG inequality, and (\ref{equ:long_open_path}), there exists a constant $c_1>0$ such that for all $k$
\[
\mathbb{P}(E_k)\geq c_1.
\]

\begin{figure}[!ht]
\centering
\includegraphics[width=6in,trim={0 8cm 2cm 6.5cm},clip]{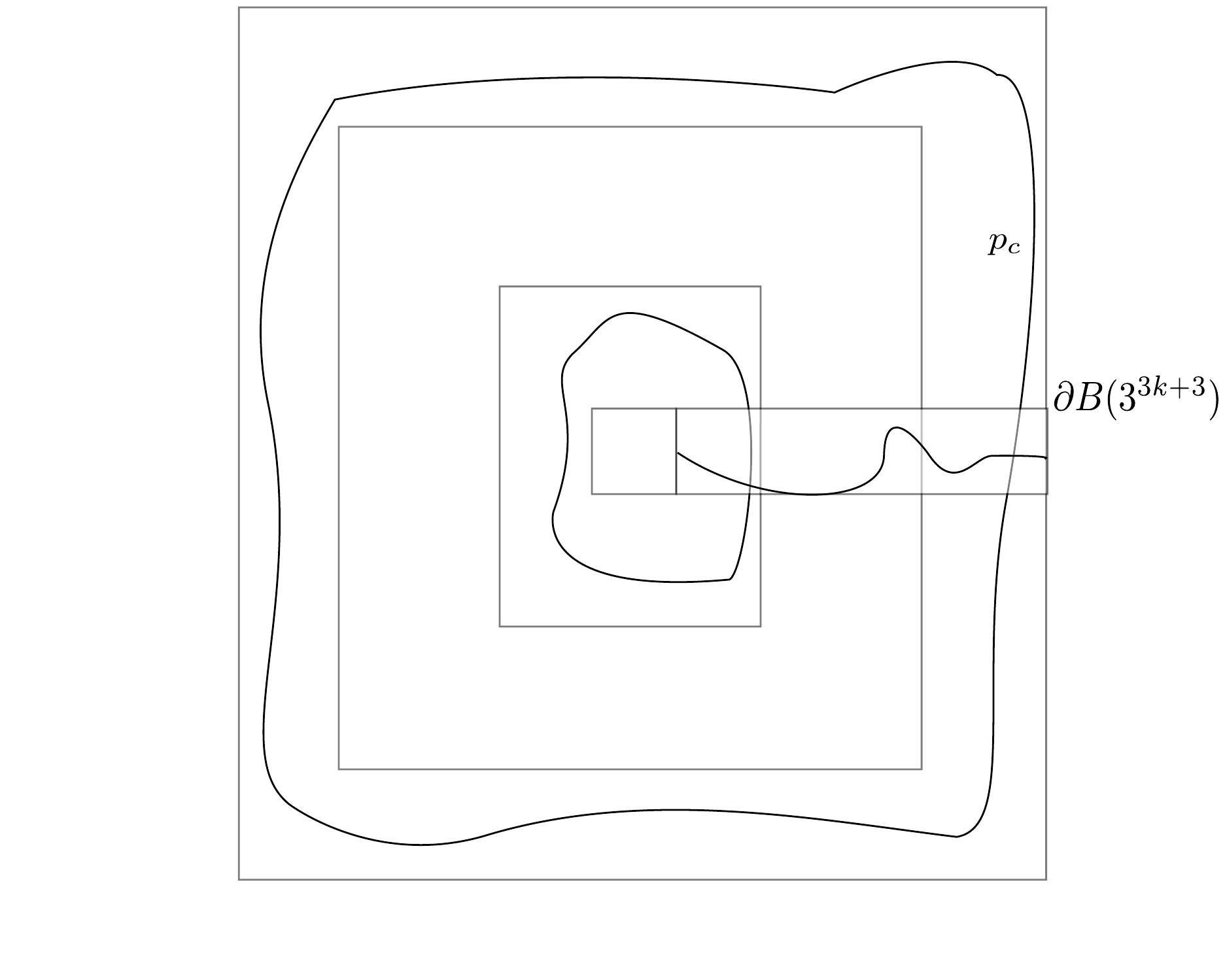}
\caption{Illustration of the event $E_k$. The boxes $B(3^{3k}), \ldots, B(3^{3k+3})$ are in increasing order. The curves represent the $p_c$-open circuits from items 1, 2, and the $p_c$-open path from item 3. By planarity, each of these curves overlap, and any geodesic from $\partial B(3^{3k})$ to $\partial B(3^{3k+3})$ must therefore cross $Ann_{k+1}$ through a $p_c$-open circuit.}
\label{fig: a_s_event}
\end{figure}

 For $k$ with $3^{3k+3}\leq n$, if $E_k$ occurs, then any geodesic from the origin to $x\notin B(n)$ must use the $p_c$-open path from item 3 to cross the annulus $Ann_{3k+1}$. Thus if $\Gamma$ is any such geodesic,
 \begin{equation}
 \# \Gamma \geq \sum_{k = 1}^{\lfloor \frac{1}{3} \log_3 n \rfloor - 1} 3^{3ks}\mathbf{1}_{E_k}.
 \end{equation}
 Note that the $E_k$'s are independent for different $k$'s. Fixing $1<s'<s$ then, there exists $c>0$ such that
 \begin{align*}
 \mathbb{P}\left( \min_{x\notin B(n)}{N_{0,x}\leq n^{s'}} \right) \leq \mathbb{P}\left(  \bigcap_{k= \lceil \frac{s'}{3s} \log_3 n \rceil + 1}^{\lfloor \frac{1}{3} \log_3 n \rfloor - 1} E_k^c  \right) &\leq (1-c_1)^{\frac{s-s'}{3s}\log_3n-2}\\
&\leq n^{-c}.
 \end{align*}
The bounds on $k$ in the above intersection are to ensure that $3^{3k+3} \leq n$ and $3^{3ks} > n^{s'}$. Applying this to $n$ which are powers of $3$ and using Borel-Cantelli, we find that a.s. for all large $m$,
\[
\min_{x\notin B(3^m)}N_{0,x}>3^{ms'}.
\]
For $x \in \mathbb{Z}^2$ with $\|x\|_\infty$ large, picking $k=\lfloor \log_3 \|x\|_\infty \rfloor - 1$, we have (noting $2\geq s>s'$)
\[
N_{0,x}\geq \min_{y\notin B(3^k)}N_{0,y}>3^{ks'} \geq c_2 \|x\|_\infty ^{s'}
\]
for some $c_2>0$. Decreasing $s'$ and relabeling it as $s$ gives \eqref{eq: easy_a_s}.	

\subsubsection{Outline of proof of Theorem~\ref{thm: main_thm}}
Here we give a brief outline of the proof of Theorem~\ref{thm: main_thm}. There are two parts: a block estimate, stated in Theorem~\ref{thm: length_thm}, and a block argument which upgrades this estimate to the stretched-exponential convergence in Theorem~\ref{thm: main_thm}. The block argument is an adaptation of the idea of Pisztora from his extension (and improvement) of ideas of Aizenman-Burchard in the near-critical percolation setting \cite{P}. Since this part is more standard, we focus here on the block estimate, Theorem~\ref{thm: length_thm}.

Two important things to notice about the statement in Theorem~\ref{thm: length_thm} are that (a) the geodesics referred to there are $T^{(n)}$-geodesics; that is, they are minimal-weight paths restricted to a box $B(3^{n+3})$ and crossing the annulus $B(3^{n+3}) \setminus B(3^n)$ (this restriction is needed in the block argument), and (b) the estimate holds almost surely for all large $n$. This second point is an upgrade from the argument given in the last subsection for \eqref{eq: easy_a_s}: the event $E_k$ there only holds with positive probability. This means that, for example, we need to control geodesic lengths even on rare events in which no $p_c$-open paths (or even near-critical paths) cross the annulus. This makes a considerable difficulty, since the arguments of Aizenman-Burchard do not apply for highly supercritical paths.

The proof of Theorem~\ref{thm: length_thm} is itself split into two parts. The first is an almost sure bound on the passage time of paths that cross annuli, stated in Proposition~\ref{prop: upper_bound}. The main estimate there is that a type of maximal passage time $T_{max}(n)$ of paths crossing the annulus $Ann_{n+1} = B(3^{n+2}) \setminus B(3^{n+1})$ satisfies
\begin{equation}\label{eq: t_max_inequality}
T_{max}(n) \leq CF^{-1}(q_n) \log^3 n \text{ for large }n,
\end{equation}
where $q_n > p_c$ is a certain near-critical percolation parameter. Similar bounds were shown in \cite{DLW} for $T(0,\partial B(n))$, the minimal passage time from $0$ to points on the boundary of $B(n)$, but we cannot use them for an annulus estimate. The reason is that if $F$ is critical and such that $T(0,x)$ is stochastically bounded in $x$, then the passage time of segments far from 0 are dominated heavily by those of segments near the origin. So we need to adapt their arguments to the annulus setting.

The proof of \eqref{eq: t_max_inequality} involves showing that with high probability, there are $q_n$-open circuits around 0 in $Ann_n$ and $Ann_{n+2}$, with a $q_n$-open path connecting them. Any geodesic crossing $B(3^{n+3}) \setminus B(3^n)$ can be modified, replacing a portion of it with the union of these $q_n$-open paths, and we obtain an upper bound for $T_{max}(n)$ by the passage time of these paths. To bound the passage time of these paths, we show in Lemma~\ref{lem: T_p} (as in \cite{DLW}) that the only edges contributing to the passage time of these paths are ones which are $q_n$-open but $p_c$-closed, and which are associated to certain $4$-arm events. Then we adapt moment bounds from \cite{kiss} for arm events in Lemma~\ref{lem: kiss_lem} to bound the number of such edges.

Once we have \eqref{eq: t_max_inequality}, we dive into the machinery of Aizenman-Burchard, analyzing the number of ``straight runs'' for geodesics. Roughly speaking, for the length of a geodesic to be linear, it must pass clear (have a straight run) through many long thin cylinders at successively decreasing scales (starting from scale $3^n$). The main inequality from Aizenman-Burchard, stated as Proposition~\ref{prop: tree}, gives a lower bound for the $s;\ell$-\emph{capacity} of a path given that it does not have too many straight runs. This capacity is related to the length of the curve in Lemma~\ref{lem: length_lemma}, and thus what we must show is that with high probability, no geodesic crossing the annulus $Ann_{n+1}$ has straight runs through cylinders at more than half of the scales of the form $L_k = \gamma^k 3^n$ from $3^n$ down to 1, where $\gamma>1$ is some large number. The difficulty here is that we do not have any precise description of geodesics in this model (including their geometry), and the only information we have is the bound on $T_{max}(n)$ above.

To show that geodesics have ``sparse'' straight runs, in Section~\ref{sec: straight_runs}, we prove that if a geodesic passes through a long thin cylinder, this cylinder has a high probability to be ``slow.'' In other words, it is likely that the cylinder is crossed in the short direction by at least 4 dual paths which are well-separated and have weight at least $p_m$, where $m$ is related to the scale of the rectangle. By choosing this $p_m$ properly, we show in Proposition~\ref{prop: straight_runs} that such a path would pass through enough slow cylinders at successive scales to have total passage time at least $(1/8)(\log^4 n)F^{-1}(q_n)$ (see \eqref{eq: q_n_lower}), where $q_n$ is the near-critical parameter above in \eqref{eq: t_max_inequality}, giving a contradiction. In short, if a geodesic does not have sparse straight runs, its passage time violates \eqref{eq: t_max_inequality}. We combine these tools in Section~\ref{sec: main_proof} to conclude the block estimate.

\section{Block estimate}\label{sec: block}

For $n \geq 1$ and $x,y \in B(3^{n+3}) = [-3^{n+3},3^{n+3}]^2$, let $T^{(n)}(x,y)$ be the minimal passage time of all paths from $x$ to $y$ that stay in $B(3^{n+3})$. Let $N_{x,y}^{(n)}$ be the minimal number of edges in any $T^{(n)}$-geodesic from $x$ to $y$. We aim to show here:
\begin{thm}\label{thm: length_thm}
There exists $s>1$ such that almost surely, the following holds for all large $n$:
\[
N_{x,y}^{(n)} \geq 3^{ns} \text{ for all } x \in B(3^n),~ y \in \partial B(3^{n+3}).
\]
\end{thm}
The proof is split into several sections. In Section~\ref{sec: cylinder_times}, we use tools from \cite{DLW} to estimate the minimal passage time across cylinders and then in Section~\ref{sec: annulus_times}, paste these together to get bounds for the passage time across annuli. In Section~\ref{sec: dimension_bounds}, we use the machinery of Aizenman-Burchard to get lower bounds on the dimension of geodesics by estimating the number of ``straight runs'' they have: if they have too few, then they are forced to go through too many edges of nonzero passage time and they violate the passage time estimates from Section~\ref{sec: annulus_times}. We bring this all together in Section~\ref{sec: main_proof} to prove Theorem~\ref{thm: length_thm}.

	

\subsection{A bound for cylinder times}\label{sec: cylinder_times}
In this section we would like to estimate the minimal passage time of any path crossing the rectangle $R(n) = [-2n,2n] \times [-n,n]$ in the first coordinate direction. So let $T(n)$ be the minimal passage time among all paths which remain in $R(n)$ and connect the left side $\{-2n\} \times [-n,n]$ to the right side $\{2n\} \times [-n,n]$. The main result of this section is a bound on $T(n)$:
\begin{prop}\label{prop: rectangle_crossing}
There exists $C>0$ such that for all $n \geq 1$ and $p > p_c$ with $L(p) \leq n$,
\[
\mathbb{P}\left(T(n) \geq \lambda F^{-1}(p) \left( \frac{n}{L(p)} \right)^2\right) \leq e^{-C\lambda} + \exp\left( -C \frac{n}{L(p)} \right) \text{ for } \lambda \geq 0.
\]
\end{prop}

For any $p$, let $E_p(n)$ be the event that there is a $p$-open path in $S(n) := [-4n,4n] \times [-n,n]$ connecting the left side $\{-4n\} \times [-n,n]$ to the right side $\{4n\} \times [-n,n]$. By \cite[Eq.~(2.8)]{jarai_03} and the RSW theorem \cite[Sec.~12.7]{grimmett}, one has the bound
\begin{equation}\label{eq: E_p_bound}
\mathbb{P}(E_p^c) \leq \exp\left( - C \frac{n}{L(p)} \right)
\end{equation}
for some $C>0$ and all $p > p_c$, $n \geq 1$. The above proposition follows immediately from this bound, \eqref{eq: hat_T_bound}, and the two lemmas below.

On $E_p(n)$, we define $T_p(n)$ as the minimal passage time among all paths which remain in $S(n)$, are $p$-open, and connect the left side to the right side. Then we put
\[
\hat T_p(n) = \max_\Gamma \sum_{e \in \Gamma \cap R(n)} t_e,
\]
where the maximum is over paths $\Gamma$ in $S(n)$ connecting the left side to the right side, which are $p$-open and have $T(\Gamma) = T_p(n)$. Note that on $E_p(n)$, one has
\begin{equation}\label{eq: hat_T_bound}
T(n) \leq \hat T_p(n).
\end{equation}

The next result characterizes the nonzero-weight edges that contribute to $\hat T_p(n)$. For $e$ with both endpoints in $R(n)$ and $p > p_c$, let $A_n(p,e)$ be the event that all of the following occur:
\begin{enumerate}
\item $\omega_e \in (p_c,p]$,
\item there are two (vertex) disjoint $p$-open paths in $S(n)$ from $e$ to $\partial B(e,n/2)$, the translate of the box $B(n/2)$ centered at the midpoint of $e$, and
\item there are two (vertex) disjoint $p_c$-closed dual paths from $e^*$ to $\partial S(n)$, one touching the top side and one touching the bottom.
\end{enumerate}
Define
\[
N_n(p) = \sum_{e \subset R(n)} \mathbf{1}_{A_n(p,e)}.
\]
\begin{lem}\label{lem: T_p}
For all $p > p_c$ and $n \geq 1$,
\[
\hat T_p(n) \mathbf{1}_{E_p(n)} \leq F^{-1}(p) N_n(p) \mathbf{1}_{E_p(n)}.
\]
\end{lem}
\begin{proof}
The proof is quite similar to \cite[Lemma~5.2]{DLW}. If the event $E_p(n)$ does not happen, then both sides are zero, the inequality holds. Thus we suppose the $E_p(n)$ happens.

Suppose $\Gamma$ is a path which remains in $S(n)$, is $p$-open, connects the left side to the right side of $S(n)$ and is such that $\hat{T}_p(n)=\sum_{e\in \Gamma\cap R(n)}{t_e}$. Note if $\omega_e\leq p_c$, then $t_e=F^{-1}(\omega_e)=0$. Hence
\[
\hat T_p(n)=\sum_{e\in\Gamma\cap R(n),\omega_e>p_c }{t_e}
\]
Since $\Gamma$ is $p$-open, we have $t_e\leq F^{-1}(p)$ for all $e\in\Gamma\cap R(n)$. Then we have
\[
\hat T_p(n)\leq F^{-1}(p)\# \{e\in\Gamma\cap R(n):\omega_e\in (p_c,p] \}
\]
Thus it suffices to show that for each $e\in \Gamma\cap R(n)$ with $p_c<\omega_e\leq p$, the event $A_n(p,e)$ occurs.

The first condition is obvious, and the second follows from that $e\in \Gamma$ and $\Gamma$ connects the left  side to the right side of $S(n)$. For the third condition, if it does not hold, then by duality we can find a $p_c$-open path that connects $\Gamma$ to itself, and the path $\Gamma'$ formed by replacing the portion of $\Gamma$ with this $p_c$-open path will avoid $e$. Since these $p_c$-open edges have zero passage time, this contradicts extremality of $\Gamma$.
\end{proof}

The next lemma gives a tail bound on the distribution of the number $N_n(p)$.
\begin{lem}\label{lem: kiss_lem}
There exists $C>0$ such that for all $n \geq 1$ and $p>p_c$ with $L(p) \leq n$,
\[
\mathbb{P}\left( N_n(p) \geq \lambda \left( \frac{n}{L(p)} \right)^2 \right) \leq e^{-C\lambda}.
\]
\end{lem}
\begin{proof}
We follow the argument of Kiss \cite[Sections~2, 3]{kiss} and much of what follows is copied from there. Since the proof is similar to that of \cite[Eq.~(3.11)]{kiss}, we just outline the main modifications necessary. For $n \geq 1$ and $p > p_c$ with $L(p) \leq n$ given, let $V_n$ be the set of edges $e$ in $R(n)$ such that $A_n(p,e)$ occurs. Let $k \in \mathbb{N}$ and
\[
X = \{e_1, \ldots, e_k\} \subset R(n).
\]
We give a bound on the probability of the event $\{V_n \supseteq X\}$, but first some definitions. Let $T_0$ denote the empty graph on $X$. Let us start blowing a box at each edge $e \in X$ at unit speed (starting at the midpoint of the edge). That is, at time $t \geq 0$, we have the boxes $B_t(e) = B(e,t),~e \in X$. We will stop at time $t=2n$, and at this time, all boxes touch.

For small values of $t$, these boxes are pairwise disjoint. As $t$ increases, more and more of these boxes intersect each other. Let $r_1$ denote the smallest $t$ when the first pair of boxes touch. We pick one such pair of boxes in some deterministic way, with central edges $e_1,f_1 \in X$. We draw an edge $\hat e_1$ between $e_1$ and $f_1$ and label it with $l(\hat e_1) := r_1$, and get the graph $T_1$. Note that $dist(e_1,f_1) = 2r_1$. (Here, $dist$ refers to the $\ell_\infty$ distance between the midpoints of the edges.) Then we continue with the growth process, and stop at time $r_2$ if we find a pair of edges $e_2,f_2 \in X$ such that $e_2,f_2$ are in different connected components of $T_1$ and $B_{r_2}(e_2)$ and $B_{r_2}(f_2)$ touch. Then we draw an edge $\hat e_2$ between one such deterministically chosen pair with the label $l(\hat e_2) := r_2$ and get $T_2$. Note that it can happen that $r_1=r_2$. We continue with this procedure until we arrive to the tree $T_{k-1}$. Let $\mathcal{R}(X)$ denote the multiset (a set where elements can appear multiple times) containing $r_i \leq 2n$ for $i=1, \ldots, k-1$.



The induction argument of \cite[Prop.~14]{KMS} implies the following product statement about our 4-arm events. To state it, we need to define a slightly modified 4-arm event. For $r \leq 2n$ and an edge $e \subset R(n)$, let $\hat\pi_4(p;e,r)$ be the probability that the following conditions hold:
\begin{enumerate}
\item $e$ is connected inside $S(n)$ to $\partial B(e,s_1)$ by two disjoint $p$-open paths, where
\[
s_1 = \min\left\{ L(p), r \right\},
\]
\item $e^*$ is connected inside $S(n)$ to $\partial B(e,s_2)$ by a $p_c$-closed dual path, where
\[
s_2 = \min\left\{ dist(e,\partial S(n)),r \right\},
\]
\item and $e^*$ is connected inside $S(n)$ to $\partial B(e,s_3)$ by another disjoint $p_c$-closed dual path, where
\[
s_3 = \min\{r,n\}.
\]
\end{enumerate}
Furthermore the paths in items 2 and 3 are alternating (open, closed, open). Set $\hat\pi_4(p;r) = \max_{e \subset R(n)} \hat \pi_4(p;e,r)$. Then for some $C_3$ independent of $n,k,p,$ and the $e_i$'s,
\begin{equation}\label{eq: inductive_bound}
\mathbb{P}(V_n \supseteq X) \leq C_3 (p-p_c)^k \hat \pi_4(p;n) \prod_{r \in \mathcal{R}(X)} \left( C_3\hat \pi_4(p;r) \right).
\end{equation}
The proof of this statement is similar to that of \cite[Prop.~2.2]{kiss}, and the main ingredient is that our connection probabilities $\hat \pi_4$ have a quasi-multiplicative property that holds for general arm events. (See the discussion around \eqref{eq: QM} above.)


Continuing from \eqref{eq: inductive_bound}, one can show that there is $C_6>0$ independent of $n,p,$ such that for $r \leq 2n$,
\begin{equation}\label{eq: four_arm_comparison}
\hat \pi_4(p;r) \leq C_6 \pi_4(s_1),
\end{equation}
where $\pi_4(s)$ is critical four-arm probability; that is, the probability of the event that the edge $f=\{0,\overrightarrow{e_1}\}$ has two disjoint $p_c$-open paths to distance $s$ (to $\partial B(f,s)$) and $f^*$ has two disjoint $p_c$-closed dual paths to distance $s$. To prove \eqref{eq: four_arm_comparison}, note that since $L(p) \leq n$, the event in $\hat \pi_4(p;e,r)$ for $e \subset R(n)$ implies that $e$ is connected to distance $s_1$ by two disjoint $p$-open paths, $e^*$ is connected to distance $s_1$ by one disjoint $p_c$-closed dual path, and $e^*$ is connected to distance $\min\{dist(e,\partial S(n)), s_1\}$ by another disjoint $p_c$-closed dual path (alternating). By independence, this probability is bounded by
\[
\pi_4'(p,\min\{dist(e,\partial S(n)),s_1\}) \pi_3^H(p,\min\{dist(e,\partial S(n)), s_1\}, s_1),
\]
where $\pi_4'(p,m)$ is the probability that $f$ is connected by two disjoint $p$-open paths to distance $m$ and $f^*$ is connected by two disjoint $p_c$-closed dual paths to distance $m$ (alternating), and $\pi_3^H(p,m_1,m_2)$ is the probability that $\partial B(m_1)$ is connected to $\partial B(m_2)$ in the upper half-plane by two disjoint $p$-open paths and a disjoint $p_c$-closed dual path (alternating). By \cite[Lemma~6.3]{DSV}, there is $D_1>0$ such that
\[
\pi_4'(p,\min\{dist(e,\partial S(n)), s_1\}) \leq D_1 \pi_4(\min\{dist(e,\partial S(n)), s_1\}).
\]
A similar argument as in \cite[Lemma~6.3]{DSV} also holds for half-plane 3-arm (annulus) events, and we find
\[
\pi_3^H(p,\min\{dist(e,\partial S(n)), s_1\},s_1) \leq D_2 \pi_3^H(\min\{dist(e,\partial S(n)),s_1\},s_1),
\]
where $\pi_3^H(m_1,m_2)$ is the probability that $\partial B(m_1)$ is connected by two disjoint $p_c$-open paths and a $p_c$-closed dual path to $\partial B(m_2)$ (alternating). Using \cite[Theorem~24 (2)]{nolin} and quasi-multiplicativity, one has for some $D_3$, $\pi_3^H(m_1,m_2) \leq D_3(m_1/m_2)^2$, and by quasi-multiplicativity and \cite[Theorem~24 (3)]{nolin}, one has $\pi_4(m_1,m_2) \geq D_4(m_1/m_2)^2$. In total, we can bound $\pi_3^H(p,m_1,m_2)$ above by a multiple of $\pi_4(m_1,m_2)$, giving uniformly in $e$, a constant $D_5$ such that
\[
\pi_4(p;e,r) \leq D_5 \pi_4(\min\{dist(e,\partial S(n)),s_1\}) \pi_4(\min\{dist(e,\partial S(n)), s_1\}, s_1),
\]
which by quasi-multiplicativity is bounded by $D_6\pi_4(s_1)$, showing \eqref{eq: four_arm_comparison}.

Given \eqref{eq: four_arm_comparison} and \eqref{eq: inductive_bound}, we obtain
\[
\mathbb{P}(V_n \supseteq X) \leq C_7(p-p_c)^k \pi_4(L(p)) \prod_{r \in \mathcal{R}(X)} \left( C_7 \pi_4(\min\{L(p),r\}) \right).
\]

To give a bound on moments of $N_n(p) = \#V_n$, we need to bound the number of sets $X$ such that $\mathcal{R}(X) = R$ for a given $R$. By arguments analogous to the proof of  \cite[Prop.~15]{KMS} we get the following. There is a universal constant $C_8$ such that for all multisets $R$ with $k-1$ elements, we have
\[
\#\{X \subset R(n) : |X| = k,~\mathcal{R}(X) = R\} \leq C_8 \mathcal{O}(R) n^2 \prod_{r \in R} (C_8 r),
\]
where $\mathcal{O}(R)$ denotes the number of different ways the elements of $R$ can be ordered.

Now compute
\begin{align*}
\mathbb{E}\binom{|V_n|}{k} &= \sum_{X \subseteq R(n)} \mathbb{P}(V_n \supseteq X) \\
&= \sum_{X \subseteq R(n)} \sum_R \mathbb{P}(V_n \supseteq X) \mathbf{1}_{\mathcal{R}(X) = R} \\
&\leq \sum_R \left[C_8 \mathcal{O}(R) n^2 C_7(p-p_c)^k \pi_4(L(p)) \prod_{r \in R} \left( C_8C_7 r \pi_4(\min\{L(p),r\}) \right)\right] \\
&\leq C_9^k n^2 (p-p_c)^k \pi_4(L(p)) \sum_R\left[ \mathcal{O}(R) \prod_{r \in R} r \pi_4(\min\{L(p),r\}) \right] \\
&= C_9^k n^2(p-p_c)^k \pi_4(L(p)) \sum_{\tilde R} \prod_{\tilde r \in \tilde R} \tilde r \pi_4(\min\{L(p),\tilde r\} \\
&= C_9^k n^2 (p-p_c)^k \pi_4(L(p)) \left( \sum_{r=1}^n r \pi_4(\min\{L(p),r\})\right)^{k-1},
\end{align*}
where $R$ is a multiset of with $k-1$ elements from the set $\{1/2, 1, \ldots, 2n\}$ and $\tilde R$ is a sequence of length $k-1$ from the set $\{1/2, 1, \ldots, 2n\}$. Last, we estimate
\[
\sum_{r=1}^n r \pi_4(\min\{L(p),r\}) \leq \sum_{r=1}^{L(p)} r \pi_4(r) + n^2 \pi_4(L(p)).
\]
For any $r \leq k$, one has $\frac{\pi_4(r)}{\pi_4(k)} \leq C_{12} (k/r)^\alpha$ for some $\alpha < 2$ (this follows from Reimer's inequality \cite{reimer} and the known value of the 5-arm exponent (from \cite[Theorem~24(3)]{nolin}, which references \cite[Lemma~5]{KSZ})), so
\begin{align*}
\sum_{r=1}^k r\pi_4(r) = \pi_4(k) \sum_{r=1}^k r \frac{\pi_4(r)}{\pi_4(k)} &\leq C_{12} \pi_4(k) \sum_{r=1}^k r (k/r)^{\alpha} \\
&= C_{12} k^{\alpha} \pi_4(k) \sum_{r=1}^k r^{1-\alpha} \\
&\leq C_{13} k^2 \pi_4(k).
\end{align*}
We thus obtain
\[
\sum_{r=1}^n r \pi_4(\min\{L(p),r\}) \leq C_{14}n^2 \pi_4(L(p)),
\]
and therefore
\[
\mathbb{E}\binom{|V_n|}{k} \leq \left( C_{15} n^2 (p-p_c) \pi_4(L(p))\right)^k.
\]
Since the product $L(p)^2 (p-p_c) \pi_4(L(p))$ is bounded uniformly in $p>p_c$ \cite[Prop~34]{nolin}, we finish with
\[
\mathbb{E}\binom{|V_n|}{k} \leq \left(C_{16} \frac{n}{L(p)}\right)^{2k}.
\]

Now we turn the above into a tail bound. For $t = 1+a > 1$,
\[
\mathbb{E}t^{|V_n|} = \sum_{k=1}^\infty (t-1)^k \mathbb{E}\binom{|V_n|}{k} \leq \sum_{k=1}^\infty \left( aC_{16} \left(\frac{n}{L(p)}\right)^2 \right)^k \leq 1,
\]
if $a$ is chosen to be $\left(2C_{16} \left(\frac{n}{L(p)}\right)^2\right)^{-1}$. This implies that for some $C_{17}>0$,
\[
\mathbb{E}\exp\left( C_{17}\frac{|V_n|}{\left(\frac{n}{L(p)}\right)^2} \right) \leq 1,
\]
and so by Markov,
\[
\mathbb{P}\left( N_n(p) \geq \lambda \left(\frac{n}{L(p)}\right)^2 \right) \leq e^{-C_{17}\lambda}.
\]
\end{proof}

\begin{cor}\label{cor: all_rectangles}
Given an integer $K \geq 2$, there exists $C>0$ such that for all $n$ and $p$ with $L(p) \leq n$,
\[
\mathbb{P}\left(T_K(n) \geq \lambda F^{-1}(p) \left( \frac{n}{L(p)} \right)^2 \right) \leq e^{-C \lambda} + \exp\left( -C \frac{n}{L(p)} \right) \text{ for } \lambda \geq 0,
\]
where $T_K(n)$ is the corresponding minimal passage time between the left and right sides of $[-Kn,Kn] \times [-n,n]$ among all paths that remain in this rectangle.
\end{cor}
\begin{proof}
Let $\Gamma_1, \ldots, \Gamma_{K-1}$ be paths such that $\Gamma_i$ is in $[-Kn + 2(i-1)n,-Kn + 2(i+1)n] \times [-n,n]$, connects the left side of the rectangle to the right side, and has minimal passage time among all such paths. Let $\hat \Gamma_1, \ldots, \hat \Gamma_{K-2}$ be paths such that $\hat \Gamma_i$ is in $[-Kn + 2in, -Kn + 2(i+1)n] \times [-n,3n]$, connects the top side of the rectangle to the bottom side, and has minimal passage time among all such paths. By planarity, there is a path remaining in $[-Kn,Kn] \times [-n,n]$ which starts on the left side of this rectangle, ends on the right, and is contained in the union $\left(\cup_{i=1}^{K-1} \Gamma_i\right) \cup \left(\cup_{i=1}^{K-2} \hat \Gamma_i \right)$. Applying Proposition~\ref{prop: rectangle_crossing} and a union bound completes the proof.
\end{proof}

\subsection{Bounds for annulus crossing times}\label{sec: annulus_times}
Using the results of the last section, we will give our main bound on types of maximal annulus crossing times.

For any $n \geq 1$, $x,y \in B(3^{n+3})$, and $S$ a subset of the edges of $B(3^{n+3})$, define
\[
T^{(n)}(x,y,S) = \max \left\{ \sum_{e \in \Gamma \cap S} t_e : \Gamma \subset B(3^{n+3}),~ T(\Gamma) = T^{(n)}(x,y),~ \Gamma: x\to y \right\}.
\]
We will be concerned with a type of annulus-crossing time. For $n \geq 2$, define
\[
T_{max}(n) = \max_{x \in \partial B(3^n), y \in \partial B(3^{n+3})} T^{(n)}(x,y, Ann_{n+1}).
\]
(Here, we think of $Ann_{n+1}:=B(3^{n+2}) \setminus B(3^{n+1}))$ as an edge set by considering all edges with both endpoints in $Ann_{n+1}$.)

The main result is:
\begin{prop}\label{prop: upper_bound}
There exist $C,\mathbf{C}$ such that almost surely
\[
T_{max}(n) \leq C F^{-1}\left( q_n \right) \log^3 n \text{ for large }n,
\]
where $q_n = p_{\lfloor \frac{3^n}{\mathbf{C}\log n} \rfloor}$.
\end{prop}
\begin{proof}
We will build two circuits around the origin -- one in $Ann_n$ and one in $Ann_{n+2}$, and a path connecting them, using only crossings in the ``long direction'' of minimal passage time of translates and rotates of rectangles of the form $[-3^{n+3},3^{n+3}]\times [-3^n,3^n]$ (they start on the left side, end on the right, and remain in the rectangle). One can do this using 9 such crossings. (Refer back to Figure~\ref{fig: a_s_event} for a similar construction. One uses 8 such crossings to build the circuits, and one to build the path connecting them.) So letting $S_n$ be the union of all the edges in these crossings, a union bound along with Corollary~\ref{cor: all_rectangles} shows for some $C, \mathbf{C}$ large enough and all $n$,
\[
\mathbb{P}\left(\sum_{e \in S_n} t_e \geq C F^{-1}(q_n) \log^3(n) \right) \leq n^{-2}.
\]
In deriving this, one needs to use \eqref{eq: L_p_n} applied to $L(q_n)$. Borel-Cantelli implies that almost surely,
\[
\sum_{e \in S_n}t_e < C F^{-1}(q_n) \log^3 n \text{ for all large }n.
\]


If $x \in \partial B(3^n)$ and $y \in \partial B(3^{n+3})$, then let $\Gamma$ be a $T^{(n)}$-geodesic from $x$ to $y$. The path $\Gamma$ has a first intersection $z$ with $S_n$ (which must be in $Ann_n$) and a last intersection $w$ with $S_n$ (which must be in $Ann_{n+2}$). One then has
\[
T^{(n)}(x,y,Ann_{n+1}) \leq T^{(n)}(z,w) \leq \sum_{e \in S_n} t_e.
\]
This is true for all $x,y$, so almost surely, for all large $n$
\[
T_{max}(n) \leq \sum_{e \in S_n} t_e \leq CF^{-1}(q_n) \log^3n.
\]
\end{proof}

\subsection{Dimension bounds}\label{sec: dimension_bounds}

\subsubsection{Tools from Aizenman-Burchard}
In this section, we recall and use some results from Aizenman-Burchard. To do this, we will think of the box $B(3^{n+3})$ scaled down to unit size; that is, to the box $B(1)$. The lattice spacing in $B(1)$ will be $1/3^{n+3}$, so that our geodesics are polygonal paths of step-size $1/3^{n+3}$. The lower bounds on dimension of random curves of \cite{AB} are derived using a truncated form of capacity.
\begin{defi}
For $s>0$ and $\ell\geq 0$, the capacity $\text{Cap}_{s;\ell}A$ of a subset $A$ of $\mathbb{R}^d$ is
\[
\frac{1}{\text{Cap}_{s;\ell}A} = \inf_{\mu \geq 0 : \int_A \text{d}\mu = 1} \int \int_{A \times A} \frac{\mu(\text{d}x)\mu(\text{d}y)}{\max\{|x-y|,\ell\}^s}.
\]
\end{defi}
(The standard definition of capacity does not include the term $\ell$ in the denominator, but this term helps to deal with the fact that our paths have stepsize $>0$.) We will take $A$ to be a geodesic $\Gamma$ in $B(1)$ from $x \in B(1/27)$ to $y \in \partial B(1)$. The relationship between the length of $\Gamma$ and its capacity is given by the following lemma.
\begin{lem}\label{lem: length_lemma}
For every collection of sets $\{B_j\}$ covering $A$ with $\min_j \text{diam}(B_j) \geq \ell$,
\[
\sum_j (\text{diam}~B_j)^s \geq \text{Cap}_{s;\ell}A.
\]
\end{lem}

Taking $\{B_j\}$ to be a collection of boxes of size $C/3^n$ centered on the edges of $\Gamma$, we then obtain
\begin{equation}\label{eq: capacity_length_bound}
\frac{\#\Gamma C^s}{3^{ns}} = \sum_j (\text{diam}~B_j)^s \geq \text{Cap}_{s;\ell}\Gamma.
\end{equation}
Therefore Theorem~\ref{thm: length_thm} follows directly from this inequality and the following proposition, which we will show in Section~\ref{sec: main_proof}:
\begin{prop}\label{prop: capacity_bound}
There exist $C_1,C_2,C_3>0$ and $s > 1$ such that the following holds for all large $n$. For all $T^{(n)}$-geodesics $\Gamma$ in $B(1)$ connecting a point in $B(1/27)$ to a point in $\partial B(1)$,
\[
\text{Cap}_{s;\ell}\Gamma' \geq C_1 \exp(-C_2 \log^4 n),
\]
where $\ell = \ell(n)$ satisfies $1/3^{n+3} \leq \ell \leq C_3/3^{n+3}$ and $\Gamma'$ is the portion of $\Gamma$ from its last intersection with $B(1/9)$ to its first intersection with $\partial B(1/3)$.
\end{prop}

\subsubsection{Geodesics have sparse straight runs}\label{sec: straight_runs}
Aizenman-Burchard gave a general theorem to lower bound the capacity of curves using the idea of ``straight runs.'' Roughly speaking, if a path does not cross too many thin cylinders (on successively decreasing scales) in the long direction, then its capacity, and therefore length, is large. We begin by recalling the definition of sparse straight runs.

Given $\gamma>1$ (which will be taken to be large), define a sequence of successive scales $L_k$ by
\[
L_k = \gamma^{-k}, \text{ for } k = 0, \ldots, k_{max},
\]
where $k_{max} = k_{max}(\gamma,n)$ is chosen so that $L_{k_{max}}$ is of order $1/3^n$. That is, we set $k_{max}$ as
\[
k_{max} = \max\{k \geq 0 : L_k \geq 1/3^{n+3}\}.
\]

\begin{defi}
A path $\Gamma$ in $B(1)$ is said to exhibit a straight run at scale $L$ ($= L_k$ for some $k$) if it traverses some cylinder of length $L$ and cross sectional diameter $(9/\sqrt{\gamma})L$ in the ``length'' direction, joining the centers of the corresponding sides. Two straight runs are \emph{nested} if one of the defining cylinders contains the other.

For a given integer $k_0$ and $\gamma>1$, we say that straight runs for $\Gamma$ are $(\gamma,k_0)$-sparse, down to the scale $\ell$, if $\Gamma$ does not exhibit any nested collection of straight runs on a sequence of scales $L_{k_1} > \cdots > L_{k_N}$ with $L_{k_N} \geq \ell$ and
\[
N \geq \frac{1}{2} \max\{k_N, k_0\}.
\]
\end{defi}

The next result says that to show our needed capacity bound for $T^{(n)}$-geodesics $\Gamma$, it suffices to prove that straight runs for $\Gamma$ are sparse. Note that in the next proposition, $\ell$ satisfies
\begin{equation}\label{eq: l_k_max}
1/3^{n+3} \leq \ell = L_{k_{max}} \leq C/3^{n+3}
\end{equation}
 for some $C = C(\gamma)$ independently of $n$, as required in Proposition~\ref{prop: capacity_bound}.
\begin{prop}\label{prop: tree}
Let $\Gamma$ be a path in $B(1)$, let $\gamma>1$ and set $m \in [\gamma/2,\gamma)$ with $\epsilon = \gamma - m$. If straight runs for $\Gamma$ are $(\gamma,k_0)$-sparse down to the scale $\ell = L_{k_{max}}$, then for $s>0$ such that $\gamma^s < \beta : = \sqrt{m(m+1)}$, one has
\[
\text{Cap}_{s;\ell}\Gamma \geq \epsilon^s \left[ \gamma^{sk_0} + \frac{\beta}{1-\beta^{-1}\gamma^s} \right]^{-1}.
\]
\end{prop}
\begin{proof}
This is the bound \cite[Eq.~(5.14)]{AB} applied in our context. (See also a similar bound and explanation above \cite[Eq.~(5.22)]{AB} with the same choice of $\beta$.)
\end{proof}

We now address sparsity of straight runs for geodesics. In the next section we use the above proposition to show Proposition~\ref{prop: capacity_bound} and conclude Theorem~\ref{thm: length_thm}. To show that geodesics must leave cylinders, we will show that many cylinders are slow in the following sense:
\begin{defi}
For $a \in (0,1)$ and $L \leq 1$, an $L \times aL$ cylinder in $B(1)$ is said to be \emph{slow} if it is traversed in the $aL$-direction by two $p_{aL3^{n+3}}$-closed dual paths $P_1,P_2$ such that
\[
\min\{|x-y| : x \in P_1, y \in P_2\} \geq aL.
\]
A cylinder that is not slow is \emph{fast}.
\end{defi}
\begin{lem}\label{lem: AB_end}
For any large enough $\gamma>1$, the following occurs almost surely for all large $n$. One cannot find a nested collection of fast cylinders $R_1, \ldots, R_N$ at scales $L_{k_1} > \cdots > L_{k_N}$ with $k_N \leq k_{max}$ and
\[
N \geq \frac{1}{4} \max\{k_N,k_0\},
\]
where $k_0 = \lceil \log^4 n \rceil$.
\end{lem}
\begin{proof}
We follow the proof in Aizenman-Burchard. We first give a bound on the probability that for any fixed sequence $k_1 < \cdots < k_N \leq k_{max}$ there is a sequence $R_1, \ldots, R_N$ of nested cylinders at scales $L_{k_1} > \cdots > L_{k_N}$ all of which are fast. Specifically, we will first show:
\begin{align}
\mathbb{P}(\text{there is a nested sequence of fast cylinders }&\text{at scales } L_{k_1}, \ldots, L_{k_N}) \nonumber \\
\leq &C_1 \gamma^{4k_N} \exp\left( - c N \sqrt{\gamma} \right) \label{eq: nested_bound}.
\end{align}

If an $L \times (9/\sqrt{\gamma})L$ cylinder is fast, then if $\gamma$ is large, (independent of $L$), we can find a cylinder of width $(10/\sqrt{\gamma})L$ and length $L/2$ centered at a line segment joining discretized points in $L'\mathbb{Z}^d$ (with $L' \leq L/\gamma$) that cannot be traversed in the $(10/\sqrt{\gamma})L$-direction by two disjoint $p_{(9/\sqrt{\gamma})L3^{n+3}}$-closed dual paths $P_1,P_2$ with $\min\{|x-y| : x \in P_1, y \in P_2\} \geq (9/\sqrt{\gamma})L$. As in \cite[Eq.~(6.2)]{AB}, the number of positions of $N$ nested cylinders at scales $L_{k_1}, \ldots, L_{k_N}$ is bounded by
\[
C_1 \gamma^{4k_1}\gamma^{4(k_2-k_1)} \cdots \gamma^{4(k_N-k_{N-1})} \leq C_1 \gamma^{4k_N}.
\]

Fix now such a sequence $R_i$, $i=1, \ldots, N$ of nested cylinders of length $L_{k_i}/2$ and width $(10/\sqrt{\gamma})L_{k_i}$. Cut each of the cylinders into $\lfloor \sqrt{\gamma}/18 \rfloor$ shorter cylinders of dimensions $(9/\sqrt{\gamma}) L_{k_i} \times (10/\sqrt{\gamma})L_{k_i}$ (plus a possible remaining one of smaller length which we do not consider) and pick a maximal number of disjoint cylinders from this collection. For $\gamma$ large, each $R_{i+1}$ intersects at most two of the shorter cylinders obtained by subdividing $R_i$, so the number of cylinders at scale $L_{k_i}$ in a maximal collection is at least $\lfloor \sqrt{\gamma}/18 \rfloor - 2$. By \eqref{eq: closed_dual_p_n}, the probability that a $(9/\sqrt{\gamma})L_{k_i} \times (10/\sqrt{\gamma})L_{k_i}$ cylinder is traversed in the $(10/\sqrt{\gamma})L_{k_i}$-direction by a $p_{(9/\sqrt{\gamma})L_{k_i}3^{n+3}}$-closed dual path is bounded below by some constant uniformly in $n, \gamma$, and the choice of the $k_i$'s. By standard large deviations for sums of Bernoulli random variables, there is a universal constant $c>0$ such that probability that at least four cylinders from the maximal collection at scale $L_{k_i}$ are traversed by such closed dual paths is at least $1-\exp(-c \sqrt{\gamma})$. If four distinct such cylinders have this property at scale $L_{k_i}$ and $\gamma$ is large, then the original cylinder $R_i$ is slow. These events are independent at distinct scales, so
\[
\mathbb{P}(R_1, \ldots, R_N \text{ are fast}) \leq \exp\left( -cN\sqrt{\gamma} \right).
\]
Summing over positions of the original cylinders gives the bound \eqref{eq: nested_bound}.

Now we sum \eqref{eq: nested_bound} over choices of cylinders to show the lemma. Namely, for a given $n$ and $k = k_N \leq k_{max}$, the probability that there is a nested sequence of fast cylinders $R_1, \ldots, R_N$ at some scales $L_{k_1} > \cdots > L_{k_N}$ with $k_N \geq N \geq k_N/4$ is bounded by
\[
\sum_{N=\lceil k/4 \rceil}^{k} \binom{k}{N} C_1 \gamma^{4k} \exp\left( - cN \sqrt{\gamma} \right) \leq C_1 k 2^k \gamma^{4k} \exp\left( - ck\sqrt{\gamma}/4 \right).
\]
Taking $\gamma$ large, this is bounded by $c_2 \exp\left( - c_2 k \right)$. Summing over $k = k_N \geq N \geq \frac{1}{4} \lceil \log^4 n \rceil$ gives a probability which is summable in $n$ and Borel-Cantelli finishes the proof.
\end{proof}

From the existence of many slow cylinders, we can prove that geodesics have sparse straight runs.
\begin{prop}\label{prop: straight_runs}
For any sufficiently large $\gamma>1$, almost surely, the following occurs for all large $n$. For any $T^{(n)}$-geodesic $\Gamma$ from a vertex $x \in B(1/27)$ to a vertex $y \in \partial B(1)$, $\Gamma'$ has $(\gamma,\lceil \log^4 n \rceil)$-sparse straight runs down to the scale $\ell = L_{k_{max}}$. Here $\Gamma'$ is the portion of $\Gamma$ from its last intersection with $B(1/9)$ to its first intersection with $\partial B(1/3)$.
\end{prop}
\begin{proof}
Take $\gamma>1$ large enough so that the event (call it $E_n$) in Lemma~\ref{lem: AB_end} holds almost surely for all large $n$. Take $\omega \in E_n$, $k_0 = \lceil \log^4 n \rceil$, and suppose for the sake of contradiction that $R_1, \ldots, R_N$ is a nested collection of cylinders at scales $L_{k_1} > \cdots > L_{k_N}$ with $k_N \leq k_{max}$ and $N \geq \frac{1}{2} \max\{k_N,k_0\}$ for which a $\Gamma'$ has straight runs in each of the $R_i$'s. Suppose that some $\lceil N/2 \rceil$ of these cylinders are fast and label them in order of decreasing scales as $R_{j_1}, \ldots, R_{j_{\lceil N/2 \rceil}}$. Then $j_{\lceil N/2 \rceil} \leq k_N \leq k_{max}$ and
\[
N/2 \geq \frac{1}{4} \max\{k_N,k_0\} \geq \frac{1}{4} \max\{j_{\lceil N/2 \rceil}, k_0\},
\]
contradicting $\omega \in E_n$. Thus at least $\lceil N/2 \rceil$ of these cylinders are slow. Let $\hat R_1, \ldots, \hat R_{\lceil N/4 \rceil}$ be the $\lceil N/4 \rceil$ slow cylinders at the smallest scales in the sequence $R_1, \ldots, R_N$. In each slow cylinder, there are two dual paths as in the definition of slow. If $\gamma$ is large enough, then each cylinder $\hat R_{i+1}$ can intersect at most one of the two closed dual paths in $\hat R_i$. Therefore, as $\Gamma'$ crosses each of these cylinders, it must intersect a distinct edge from at least one dual path in each $\hat R_i$. This means that if $e_1, \ldots, e_{N/4}$ are such edges, then $T(\Gamma') \geq t_{e_1} + \cdots + t_{e_{\lceil N/4 \rceil}}$. If the cylinder $\hat R_i$ is at scale $L_{\hat k_i}$, then $\hat k_i \geq \lfloor N/4 \rfloor$, so because $\hat R_i$ is slow,
\[
t_{e_i} \geq F^{-1}\left( p_{(9/\sqrt{\gamma})L_{\hat k_i} 3^{n+3}}\right) \geq F^{-1}\left(p_{(9/\sqrt{\gamma})L_{\lfloor N/4 \rfloor}3^{n+3}}\right).
\]
As $N \geq k_0/2 \geq  \frac{1}{2}\lceil \log^4 n \rceil$, one has for large $n$
\[
\frac{9}{\sqrt{\gamma}} L_{\lfloor N/4 \rfloor} 3^{n+3} \leq \left\lfloor \frac{3^n}{\mathbf{C} \log n} \right\rfloor,
\]
where $\mathbf{C}$ is from Proposition~\ref{prop: upper_bound}. Therefore
\begin{equation}\label{eq: q_n_lower}
T(\Gamma') \geq \lceil N/4 \rceil F^{-1}(p_{\lfloor \frac{3^n}{\mathbf{C} \log n} \rfloor}) \geq \frac{1}{8} (\log^4 n) F^{-1}(q_n),
\end{equation}
where $q_n = p_{\lfloor \frac{3^n}{\mathbf{C} \log n} \rfloor}$, and this contradicts Proposition~\ref{prop: upper_bound} for large $n$.
\end{proof}

\subsection{Proof of Theorem~\ref{thm: length_thm}}\label{sec: main_proof}

We prove Theorem~\ref{thm: length_thm} by showing Proposition~\ref{prop: capacity_bound}. As before, we shrink the lattice so that $B(3^{n+3})$ is shrunk to $B(1)$. Choose $\gamma>1$ large enough for Proposition~\ref{prop: straight_runs} and so that if $m = \lfloor \gamma \rfloor$, then $\sqrt{m(m+1)} > \gamma$. This means in particular that if we define $s$ by $\gamma^s = \frac{\sqrt{m(m+1)} + \gamma}{2}$, then $s > 1$.

We now apply Proposition~\ref{prop: straight_runs} along with the capacity lower bound Proposition~\ref{prop: tree}. The first says that if $\Gamma$ is a $T^{(n)}$-geodesic from a vertex $x \in B(1/27)$ to a vertex $y \in \partial B(1)$, then $\Gamma'$ has $(\gamma, \lceil \log^4 n \rceil)$-sparse straight runs down to scale $\ell = L_{k_{max}}$, a number satisfying \eqref{eq: l_k_max}. In this setting, Proposition~\ref{prop: tree} gives for $\epsilon = \gamma - m$, $L_0=1$, and $\beta = \sqrt{m(m+1)}$,
\[
\text{Cap}_{s;\ell} \Gamma' \geq \epsilon^s \left[ \gamma^{s \lceil \log^4 n \rceil} + \frac{\beta}{1-\beta^{-1} \gamma^s} \right]^{-1} \geq C_1 \exp(-C_2 \log^4 n).
\]
This proves Proposition~\ref{prop: capacity_bound}.

Last we combine this bound with equation \eqref{eq: capacity_length_bound}, applied to $\Gamma'$. Putting $\{B_j\}$ as a collection of boxes of size $C_3/3^{n+3}$ centered on the edges of $\Gamma'$ (where $C_3$ is from Proposition~\ref{prop: capacity_bound}), one has $\# \Gamma \geq \frac{C_1}{C_3^s} 3^{ns} \exp\left(-C_2 \log^4 n \right)$, so
\[
\min_{x \in B(3^n), y \in \partial B(3^{n+3})} N_{x,y}^{(n)} \geq  \frac{C_1}{C_3^s} 3^{ns} \exp(-C_2 \log^4 n).
\]
Since $\exp(C_2 \log^4 n) = o(3^{\delta n})$ for each $\delta>0$, we can slightly decrease $s > 1$ to obtain Theorem~\ref{thm: length_thm}.

\section{Block argument: proof of Theorem~\ref{thm: main_thm}}

	


Let $m$ be such that $3^{m-1} \leq \|x\|_\infty < 3^m$ and set $s>1$ as the constant in Theorem~\ref{thm: length_thm}. Fix $s' \in (1,s)$, and let $n=\lceil m \frac{s'}{s}\rceil$, so that $3^{ns}\geq 3^{ms'}$.

\begin{defi}\label{good annulus}
For $y\in \mathbb{Z}^2$, define the annulus $A(y,n):=2 \cdot 3^n y+B(3^{n+3} ) \setminus B(3^n)$. For $z,w \in y+B(3^{n+3})$, define $T_y^{(n)}(z,w)$ as the minimal passage time from $z$ to $w$ among paths remaining in $y+B(3^{n+3})$ and $N_y^{(n)}(z,w)$ the minimal number of edges in any $T_y^{(n)}$-geodesic from $z$ to $w$. Call $A(y,n)$ bad if
\[
\min_{z\in y+\partial B(3^n),w\in y+\partial B(3^{n+3})}N_y^{(n)}(z,w)< 3^{ns}
\]
\end{defi}
By stationarity, $\hat p_n := \mathbb{P}(A(y,n) \text{ is bad})$ depends only on $n$, and by Theorem~\ref{thm: length_thm}, it approaches 0 as $n \to \infty$.


For a geodesic $\Gamma$ from $0$ to $x$, we may follow $\Gamma$, marking each box of the form $y+B(3^n)$ that it touches inside the box $B(3^m)$. Note that if $\|x\|_\infty$ is large enough, $\Gamma$ must cross the annulus $A(y,n)$ surrounding the box. By standard arguments, we can extract a sequence $\gamma = (A_1, \ldots, A_r)$ of these ``crossed'' annuli satisfying the following properties: for universal constants $c_1,c_2 > 0$,
\begin{enumerate}
\item $A_i$ and $A_j$ are disjoint for $i \neq j$,
\item if $|i-j|=1$ and $A_i = A(y_i,n)$, $A_j = A(y_j,n)$, then $\|y_i - y_j\|_\infty \leq c_1$,
\item $\|y_1\|_\infty \leq c_1$, and
\item $r \geq c_2 3^{m-n}$.
\end{enumerate}

Note that if any one of these annuli $A(y,n)$ is not bad, then defining $z$ to be the first entrance of $\Gamma$ to $y+B(3^n)$ and $w$ the last entrance of $\Gamma$ to $y+B(3^{n+3})$ before $z$, then
\[
\# \Gamma \geq N_y^{(n)}(z,w) \geq 3^{ns} \geq 3^{ms'}.
\]
Hence, letting $c_3>0$ be such that, given $y$, there are at most $c_3$ choices of $y'$ with $\|y-y'\|_\infty \leq c_1$, one has



\begin{align*}
\mathbb{P}(N_{0,x} < \|x\|_{\infty}^{s'}) \leq \mathbb{P}(N_{0,x} < 3^{ms'}) &\leq  \sum_{r \geq c_2 3^{m-n}} \sum_{\#\gamma =r}\mathbb{P}(\text{all } A_i \in \gamma \text{ are bad}) \\
&\leq  \sum_{r \geq c_2 3^{m-n}} \sum_{\# \gamma =r} \hat{p}_n^r\\
&\leq \sum_{r \geq c_2 3^{m-n}} (c_3 \hat{p}_n)^r
\end{align*}
By  Theorem~\ref{thm: length_thm}, there exists constant $N>0$ such that when $n\geq N$, $c_3 \hat{p}_n \leq 1/2$. So choosing $\|x\|_\infty$ large enough so that $n \geq N$, we obtain
constants $c_4,c_5>0$ such that
\begin{equation}
\mathbb{P}(N_{0,x} < \|x\|_\infty^{s'})\leq  \sum_{r \geq c_2 3^{m-n}} 2^{-r} \leq c_4e^{-c_5 3^{m-n}}.
\end{equation}
This implies Theorem~\ref{thm: main_thm}, since $3^{m-1} \leq \|x\|_\infty < 3^m$.


\begin{thebibliography}{9}

\bibitem{AB}
M. Aizenman and A. Burchard. (1999). H\"older regularity and dimension bounds for random curves. \emph{Duke Math. Journal.} {\bf 99}, 419-453.

\bibitem{ADH15}
A. Auffinger, M. Damron, and J. Hanson. (2015). 50 years of first-passage percolation. \emph{arXiv}:1511.03262.

\bibitem{DHS_RW}
M. Damron, J. Hanson, and P. Sosoe. (2013). Subdiffusivity of random walk on the 2$D$ invasion percolation cluster. {\t Stoch. Proc. Appl.} {\bf 123}, 3588-3621.

\bibitem{DLW}
M. Danron, W. Lam and X. Wang. (2015). Asymptotics for 2D critical first passage percolation. To appear in \emph{Ann. Probab.}

\bibitem{DSV}
M. Damron, A. Sapozhnikov, and B. V\'agv\"olgyi. (2009). Relations between invasion percolation and critical percolation in two dimensions. \emph{Ann. Probab.} {\bf 37}, 2297-2331.

\bibitem{grimmett}
G. R. Grimmett. \emph{Percolation}. 2nd edition, Grundlehren der matematischen Wissenschaften 321. Berlin: Springer, 1999.

\bibitem{HW}
J. Hammersley and D. Welsh. First-passage percolation, subadditive processes, stochastic networks, and generalized renewal theory. 1965 {\it Proc. Internat. Res. Semin., Statist. Lab., Univ. California, Berkeley, Calif.,} 61-110, {\it Springer-Verlag, New York}.

\bibitem{jarai_03}
A. A. J\'arai. (2003). Invasion percolation and the Incipient Infinite Cluster in $2D$. \emph{Commun. Math. Phys.} {\bf 236}, 311-334.

\bibitem{aspects}
H. Kesten. Aspects of first-passage percolation. \emph{\'Ecole d'\'et\'e de probabilit\'es de Saint-Flour, XIV--1984,} 125-264, Lecture Notes in Math., 1180, \emph{Springer, Berlin,} 1986.

\bibitem{kesten_scaling_relations}
H. Kesten. (1987). Scaling relations for $2D$-percolation. \emph{Commun. Math. Phys.} {\bf 109}, 109-156.

\bibitem{KSZ}
H. Kesten, V. Sidoravicius, and Y. Zhang. (1998). Almost all words are seen in critical site percolation on the triangular lattice. \emph{Electron. J. Probab.} {\bf 3}, paper no. 10.

\bibitem{kiss}
D. Kiss. (2014). Large deviation bounds for the volume of the largest cluster in $2D$ critical percolation. \emph{Electron. Commun. Probab.} {\bf 19}, 1-11.

\bibitem{KMS}
D. Kiss, I. Manolescu, and V. Sidoravicius. (2015). Planar lattices do not recover from forest fires. \emph{Ann. Probab.} {\bf 43}, 3216-3238.

\bibitem{nolin}
P. Nolin. (2008). Near-critical percolation in two dimensions. \emph{Electron. J. Probab.} {\bf 13}, 1562-1623.

\bibitem{P}
A. Pizstora. Scaling inequalities for shortest paths in regular and invasion percolation. Carnegie-Mellon CNA preprint. Available at http://www.math.cmu.edu/CNA/Publications/publications2000/001abs/00-CNA-001.pdf.

\bibitem{reimer}
D. Reimer. (2000). Proof of the van den Berg-Kesten conjecture. \emph{Combin. Probab. Comput.} {\bf 9}, 27-32.

\bibitem{SZ}
J. Michael Steele and Y. Zhang. (2003). Nondifferentiability of the time constants of first-passage percolation. {\it Ann. Probab.} {\bf 31}, 1028-1051.


\bibitem{WR}
J. Wierman and W. Reh. (1978). On conjecture in first passage percolation theory. \emph{Ann. Probab.} {\bf 6}, 388-397.

\bibitem{Zhang95}
Y. Zhang. (1995). Supercritical behaviors in first-passage percolation. \emph{Stoch. Proc. Appl.} {\bf 59}, 251-266.

\bibitem{Zhang99}
Y. Zhang. Double behavior of critical first-passage percolation. \emph{Perplexing problems in probability,} 143-158, Progr. Probab., 44, \emph{Birkh\"auser Boston, Boston, MA, 1999}.

\end{thebibliography}
\end{document}